\documentclass[amssymb,12pt]{amsart}

\usepackage{latexsym}

\usepackage{hyperref}
\usepackage{amssymb,amsfonts,amsmath,amsopn,amstext,amscd,latexsym}
\usepackage{bbm}
\usepackage[all,cmtip]{xy}

\usepackage{enumerate}
\usepackage[shortlabels]{enumitem}

\newtheorem{propo}{Proposition}[section]
\newtheorem{prop}[propo]{Proposition}

\newtheorem{lem}[propo]{Lemma}
\newtheorem{cor}[propo]{Corollary}

\newtheorem{thm}[propo]{Theorem}

\newcommand{\Ker}{\operatorname{Ker}}

\newcommand{\Irr}{{\mathrm {Irr}}}

\def\rank{\mathop{\mathrm{ rank}}\nolimits}

\newcommand{\RR}{{\mathbb R}}
\newcommand{\QQ}{{\mathbb Q}}
\newcommand{\ZZ}{{\mathbb Z}}
\newcommand{\NN}{{\mathbb N}}

\newcommand{\SSS}{{\sf S}}

\newcommand{\AAA}{{\sf A}}
\newcommand{\FF}{{\mathbb F}}

\newcommand{\SN}{\SSS_{n}}
\newcommand{\AN}{\AAA_{n}}

\begin{document}

\title[Refined Waring]
{A Refined Waring Problem for Finite Simple Groups}

\author{Michael Larsen}
\address{Department of Mathematics\\
    Indiana University \\
    Bloomington, IN 47405\\
    U.S.A.}
\email{mjlarsen@indiana.edu}

\author{Pham Huu Tiep}
\address{Department of Mathematics\\
    University of Arizona\\
    Tucson, AZ 85721\\
    U. S. A.} 
\email{tiep@math.arizona.edu}

\thanks{Michael Larsen was partially supported by NSF 
grant DMS-1101424 and a Simons Fellowship.  Pham Huu Tiep was partially supported by
NSF grant  DMS-1201374.}

\maketitle

\begin{abstract}
Let $w_1$ and $w_2$ be nontrivial words in free groups $F_{n_1}$ and $F_{n_2}$ 
respectively. We prove that, for all sufficiently large finite non-abelian simple groups $G$, there exist subsets $C_1 \subseteq w_1(G)$ and $C_2 \subseteq w_2(G)$ such that 
$|C_i| = O(|G|^{1/2}\log^{1/2} |G|)$ and $C_1C_2 = G$. In particular, if 
$w$ is any nontrivial word and $G$ is a sufficiently large finite non-abelian simple group,
then $w(G)$ contains a thin base of order $2$. This is a non-abelian analogue of 
a result of Van Vu \cite{Vu} for the classical Waring problem. Further results concerning
thin bases of $G$ of order 2 are established for any finite group and for any compact Lie group $G$.  
\end{abstract}
\section{Introduction}

Let $F_n$ denote the free group in $n$ generators and $w\in F_n$ a nontrivial element.
For every group $G$, the word $w$ induces a function $G^n\to G$, which we also denote $w$.
In joint work with Aner Shalev \cite{LS2,LST}, the authors proved that if $G$ is a finite simple group whose order is sufficiently large in terms of $w$, then $w(G^n)$ is a \emph{basis of order $2$}, i.e.,
every element of $G$ can be written as the product of two elements of $w(G^n)$.  In particular,
for any positive  integer $m$, the $m$th powers in $G$ form a basis of order $2$ for all sufficiently large finite simple groups; this example explains the use of the term ``Waring problem''
in the title of this paper.

The refinement we have in mind is indicated by a result of Van Vu \cite{Vu} on  the classical Waring problem.  Vu observed that the $m$th powers in the set $\NN$ of natural numbers form a 
\emph{thick} basis of sufficiently large order $s$, in the sense that the number of representations
of $n\in \NN$ as a sum of $s$ $m$th powers grows polynomially with $n$.  He proved that the $m$th powers contain \emph{thin} subbases of order $s$, i.e. subsets $X$ for which every element of $\NN$ can be written as a sum of $s$ elements of $X$, but the growth of the number of representations is logarithmic.  He asked one of us if there is an analogous result in the group-theoretic setting, i.e., if $w(G^n)$ contains a thin subbase of order $2$.  The main result of this paper gives an affirmative answer to this question; in fact, the growth of the average number of representations of $g\in G$  is
$O(\log|G|)$.  

More precisely, our result is as follows.  We state it asymmetrically, i.e. in the more general case that
we have two possibly different words $w_1$ and $w_2$ instead of a single word $w$.

\begin{thm}
\label{main}
Let $w_1$ and $w_2$ be nontrivial words in free groups $F_{n_1}$ and $F_{n_2}$ respectively.
For all  sufficiently large finite non-abelian simple groups $G$, there exist subsets 
$C_1\subseteq w_1(G)$ and $C_2\subseteq w_2(G)$
such that $|C_i| = O(|G|^{1/2}\log^{1/2} |G|)$ and $C_1 C_2 = G$.
\end{thm}

It is known that for many words $w$, we have $w(G^n) = G$ for all $G$ sufficiently large.  For instance, the commutator word in $F_2$ satisfies this equality for all finite simple $G$, cf. \cite{EG}, \cite{LBST}.  In this case, we are looking for a thin subbase of $G$ itself, and we prove that such order $2$ subbases $X_G$ exist, not merely for finite simple groups but for all finite groups, where the average number of representations of $G$ as a product of two elements in $X_G$
is $O(1)$ as $|G|\to \infty$, see Corollary \ref{root}.  We
conclude with an analogous result for compact Lie groups, cf. Proposition \ref{compact1} and
Theorem \ref{compact2}.

\section{The Probabilistic Method}

Given subsets $X$ and $Y$ of a finite group $G$ with $XY = G$, we would like to find subsets $X_0\subseteq X$ and $Y_0\subseteq Y$ such that $X_0 Y_0$ is still all of $G$,
while $|X_0| |Y_0|$ is only
slightly larger than $|G|$.  
In this section we show that  appropriately large random subsets $X_0\subseteq X$ and $Y_0\subseteq Y$  
usually have the property that $X_0 Y_0$ includes every element of $G$ that has many representations of the form $xy$, $x\in X$, $y\in Y$.

\begin{lem}
\label{empty}
Let  $a,b,n$ be positive integers,  $N$ a set of cardinality $n$, $A\subseteq N$ a fixed subset of cardinality $a$, and $B\subseteq N$ a random subset chosen uniformly
from all $b$-element subsets of $N$.  Then 
$$\Pr[A\cap B = \emptyset] \le e^{-ab/n}.$$
\end{lem}

\begin{proof}
The statement is trivial if $a+b>n$, so we assume $a+b\le n$.
The probability that $A\cap B = \emptyset$ is
\begin{align*}
\frac{\binom {n-a}b}{\binom nb} &= \frac{(n-a)! \, (n-b)!}{n! \, (n-a-b)!} = \frac{(n-a)(n-a-1)\cdots (n-a-b+1)}{n(n-1)\cdots (n-b+1)} \\
                      &\le (1-a/n)^b \le e^{-ab/n}. \\
\end{align*}

\end{proof}

The following lemma gives a somewhat cruder but more general estimate than Lemma~\ref{empty}.

\begin{lem}
\label{too-small}
Let  $a,b,n$ be positive integers,  $N$ a set of cardinality $n$, $A\subseteq N$ a fixed subset of cardinality $a$, and $B\subseteq N$ a random subset chosen uniformly
from all $b$-element subsets of $N$.  Then 
$$\Pr\Bigl(|A\cap B| \le \frac{ab}{e^2 n}\Bigr) \le (2.2)e^{-\frac{5ab}{2e^2 n}}.$$
\end{lem}

\begin{proof}
Assume $\max(a+b-n,0)\le k \le \min(a,b)$ so that $k$ is a possible size for $A\cap B$. For $k > 0$ we have $k! > (k/e)^k$ and so 
the probability that $|A\cap B| = k$ is
\begin{align*}
\frac{\binom ak \binom{n-a}{b-k}}{\binom nb} &= \frac{a! \, b! \, (n-a)! \, (n-b)!}{k! \, (a-k)! \, (b-k)! \, n! \, (n-a-b+k)!} \\
                   &=\frac{b\cdots (b-k+1)}{k!}\frac{a\cdots(a-k+1)}{n\cdots (n-k+1)}\frac{(n-a)\cdots(n-a-b+k+1)}{(n-k)\cdots (n-b+1)} \\
                  & < \frac{b^k}{(k/e)^k} \frac{a^k}{n^k} \frac{(n-a)^{b-k}}{(n-k)^{b-k}} \le \frac{(ab/n)^k}{(k/e)^k} \exp(-\frac{(b-k)(a-k)}{n-k})= \exp(f(k)), \\
\end{align*}
where
$$f(x) := x+x\log ab/n - x\log x - g(x),~~g(x) := (a-x)(b-x)/(n-x).$$
Let $r :=  ab/e^2n \leq \min(a/e^2,b/e^2)$. 
Then, when $0 < x \leq r$ we have $f'(x) > 2$, and so $f(x)$ is increasing
on $(0,r]$ and $f(x)-f(x-1) > 2$ when $1 < x \leq r$. Also,
$$g(r) \geq \frac{ab(1-e^{-2})^2}{n} > 5.5r,~~~f(r) = 3r -g(r) < -2.5r$$
It follows that     
$$\Pr(0 < |A\cap B| \le r) \le \sum_{i=1}^{\lfloor r\rfloor} \exp(f(i)) < \frac 1{1-e^{-2}} \exp(f(r)) < \frac{e^{-2.5r}}{1-e^{-2}} < (1.2)e^{-2.5r}.$$
Together with Lemma \ref{empty}, this implies the claim.
\end{proof}

\begin{prop}
\label{main-prop}
Let  $c > 0$ be a constant and $X$, $Y$, and $Z$ subsets of a finite group $G$ such that for all $z\in Z$,
$$|\{(x,y)\in X\times Y\mid xy = z\}| \ge \frac{c|X|\,|Y|}{|G|}.$$
Let $x_0 \le |X|$ and $y_0\le |Y|$ be positive integers such that 
$$x_0 y_0 \geq (2e^2/c) |G| \log |G|.$$
Then there exist subsets $X_0\subseteq X$ and $Y_0\subseteq Y$, with $x_0$ and $y_0$ elements respectively,  such that $X_0 Y_0 \supseteq Z$.
\end{prop}

\begin{proof}
Let $n$ denote the order of $G$, which we may assume is at least $2$.
We choose $X_0$ and $Y_0$ at random independently and uniformly from the subsets of $X$ of cardinality $x_0$ and the subsets of $Y$ of
cardinality $y_0$ respectively.
It suffices to prove that for each $z\in Z$, the probability that $z \in X_0 Y_0$ is more than $1-1/n$.
(Indeed, in this case the probability that $X_0Y_0 = G$ is larger than $1-n/n = 0$, i.e. 
$X_0Y_0 = G$.)
Let $S_z$ denote the set of pairs $(x,y)\in X\times Y$ such that $xy=z$, and let $\pi_X$, $\pi_Y$ denote the projection maps
from $X\times Y$ to $X$ and $Y$ respectively.  
We want to prove that the probability that $\pi_Y^{-1}(Y_0)\cap \pi_X^{-1} (X_0)\cap S_z$ is non-empty is more than $1-1/n$.

As $G$ is a group, the restrictions of $\pi_X$ and $\pi_Y$ to $S_z$ are injective, so 
$$|\pi_X^{-1} (X_0)\cap S_z| = |\pi_X (S_z)\cap X_0|,\ |\pi_Y^{-1}(Y_0)\cap \pi_X^{-1} (X_0)\cap S_z| = |\pi_Y(\pi_X^{-1} (X_0)\cap S_z)\cap Y_0|.$$
It suffices to prove that the probability that $\pi_X (S_z)\cap X_0$ has at least $\frac{x_0 |S_z|}{e^2 |X|}$ elements is at least $1-1/2n$
and that the conditional probability that $\pi_Y(\pi_X^{-1} (X_0)\cap S_z)\cap Y_0$ is non-empty given that 
\begin{equation}
\label{condition}
|\pi_X (S_z)\cap X_0| \ge \frac{x_0 |S_z|}{e^2 |X|}
\end{equation}
is at least $1-1/2n$.

By hypothesis,
$$\frac{|X_0| |\pi_X(S_z)|} {|X|} = \frac{x_0 |S_z|} {|X|} \ge \frac{c x_0  |Y|} n \ge \frac{c x_0 y_0}n \geq 2e^2 \log n.$$
By Lemma~\ref{too-small}, the probability that 
$$|X_0 \cap \pi_X(S_z)| = |\pi_X^{-1} (X_0)\cap S_z|  \le \frac{x_0 |S_z|} {e^2|X|}$$
is at most $2.2/n^5 < 1/2n$.  
If (\ref{condition}) holds, then
$$\frac{|Y_0| |\pi_X^{-1} (X_0)\cap S_z| }{|Y|} \ge \frac{x_0 y_0 |S_z|} {e^2|X||Y|} \ge \frac{2 n \log n |S_z|}{c |X| |Y|} \ge 2\log n.$$
By Lemma~\ref{empty}, the probability of $Y_0$ being disjoint from a subset of $Y$ of cardinality at least $\frac{x_0 |S_z|} {e^2|X|}$
is at most $1/n^2 \leq 1/2n$.
\end{proof}

\begin{cor}\label{subset}
Let $w_1$, $w_2$ be two nontrivial words and $S$ a finite simple group. To prove Theorem \ref{main} for $(w_1, w_2, S)$, it suffices to show that there exist subsets
$X \subseteq w_1(S)$, $Y \subseteq w_2(S)$ and a subset 
$S_1 \subset S$ of cardinality at most $|S|^{1/2}$, such that 

\begin{enumerate}[\rm(i)]
\item $w_1(S)w_2(S) = S$;

\item $|\{(x,y) \in X \times Y \mid xy =g\}| \geq \dfrac{|X| \cdot |Y|}{2|S|}$   
for all $g \in S \setminus S_1$;

\item $|X|, |Y| \geq 2e|S|^{1/2}\log^{1/2}|S|$.
\end{enumerate}
\end{cor}

\begin{proof}
Choose $x_0 = y_0 := \lfloor 2e|S|^{1/2}\log^{1/2}|S| \rfloor$ (note that 
we still have $x_0 \leq |X|$ and $y_0 \leq |Y|$). By Proposition \ref{main-prop} with $c=1/2$, there exist
subsets $X_0 \subseteq X$ and $Y_0 \subseteq Y$ with $X_0Y_0 \supseteq S \setminus S_1$,
$|X_0| = x_0$, and $|Y_0| = y_0$. For each $z \in S_1$, by (i) there exists 
$(x_z,y_z) \in w_1(S) \times w_2(S)$ such that $z = x_zy_z$. Now set
$$C_1 := X_0 \cup \{x_z \mid z \in S_1\},~~C_2 := Y_0 \cup \{y_z \mid z \in S_1\}.$$
\end{proof}

\begin{cor}
If  $x_0$ and $y_0$ are  integers in $[1,|G|]$ such that $x_0 y_0 > 2e^2 |G| \log |G|$, then there exist
subsets $X_0$ and $Y_0$ of $G$ of cardinality $x_0$ and $y_0$ respectively such that $X_0 Y_0 = G$.
\end{cor}

\begin{proof}
Set $X = Y = Z := G$ and $c=1$ in Proposition \ref{main-prop}.
\end{proof}

\begin{cor}
There exists a \emph{square root} $R$ of $G$, i.e., a subset such that $R^2 = G$, with 
$|R| \le 2^{1/2}e |G|^{1/2} \log^{1/2}|G|$.
\end{cor}

In fact, we will show that $G$ has a square root of size $O(|G|^{1/2})$, see Corollary \ref{root}.
Analogues of this result for compact Lie groups will be proved in \S6, see Proposition
\ref{compact1} and Theorem \ref{compact2}.

\section{Simple Groups of Lie Type}

\newcommand{\cG}{\mathcal{G}}
\newcommand{\cT}{\mathcal{T}}
\newcommand{\rmO}{\mathrm{O}}
\newcommand{\bfZ}{\mathbf{Z}}
\newcommand{\bbZ}{\mathbb{Z}}
\newcommand{\bfC}{\mathbf{C}}
\newcommand{\St}{{\sf{St}}}
\newcommand{\tw}[1]{{}^#1}

In what follows, we say that $S$ is a finite simple group of Lie type of rank $r$ defined over $\FF_q$, if 
$S = \cG^F/\bfZ(\cG^F)$ for a simple, simply connected algebraic group $\cG$ over $\FF_q$,
of rank $r$, and a Steinberg endomorphism $F:\cG \to \cG$, with $q$ the common absolute 
value of the eigenvalues of $F$ on the character group of an $F$-stable maximal torus $\cT$ of 
$\cG$. In particular, this includes the Suzuki-Ree groups, for which $q$ is a half-integer power 
of $2$ or $3$. By slight abuse of terminology, we will say that an element $s\in S$ 
is regular semisimple if some inverse image of $s$ is so in $\cG^F$.

\smallskip
The aim of this section is to prove the following theorem:

\begin{thm}\label{main1}
Let $w_1$, $w_2$ be two nontrivial words. Then there is $N = N(w_1,w_2)$ with the following 
property. For any finite non-abelian simple group $S$ of Lie type of order at least $N$, there exist conjugacy classes $s_1^S \subseteq w_1(S)$, $s_2^S \subseteq w_2(S)$ and a subset 
$S_1 \subset S$ of cardinality at most $|S|^{1/2}$, such that 
\begin{enumerate}[\rm(i)]
\item $w_1(S)w_2(S) = S$;

\item $|\{(x,y) \in s_1^S \times s_2^S \mid xy =g\}| \geq \dfrac{|s_1^S| \cdot |s_2^S|}{2|S|}$   
for all $g \in S \setminus S_1$;

\item $|s_i^S| \geq 4e|S|^{1/2}\log^{1/2}|S|$.
\end{enumerate}
\end{thm}
 
Note that condition (i) follows from the main result of \cite{LST}, and (ii) is equivalent to
\begin{equation}\label{bound1}
  \left|\sum_{1_S \neq \chi \in \Irr(S)}\frac{\chi(s_1)\chi(s_2)\bar\chi(g)}{\chi(1)}\right| \geq \frac{1}{2},
  ~~\forall g \in S \setminus S_1.
\end{equation}
Also, Theorem \ref{main1} and Corollary \ref{subset} immediately imply Theorem 
\ref{main} for sufficiently large non-abelian simple groups of Lie type.
 
\smallskip
First we recall the following consequence of \cite[Proposition 7]{La}:

\begin{lem}\label{image}
For any $r_0$ and any nontrivial word $w \neq 1$, there exists a constant $c = c(w,r_0)$ such that
$$|w(S)|  \geq c|S|$$ 
for all finite simple group $S$ of Lie type of rank $\leq r_0$.
\end{lem}

\begin{cor}\label{regular}
For any $r_0$ and any nontrivial word $w \neq 1$, there exists a constant $Q = Q(w,r_0)$ such that

\begin{enumerate}[\rm(i)]

\item $w(S)$ contains a regular semisimple element $s$, and 

\item $|x^S| \geq 4e|S|^{1/2} \log^{1/2}|S|$ for any regular semisimple element $x \in S$,
\end{enumerate}
for all finite simple groups $S$ of Lie type of rank 
$\leq r_0$ defined over $\FF_q$ with $q \geq Q$.
\end{cor}

\begin{proof}
According to \cite[Theorem 1.1]{GL}, the proportion of regular semisimple elements in $S$ defined over $\FF_q$ is more than $1-f(q)$, with 
$$f(q) := \frac{3}{q-1}+\frac{2}{(q-1)^2}.$$
Applying Lemma \ref{image} and choosing $Q$ so that $f(Q) < c(w,r_0)$, we see that 
$w(S)$ contains a regular semisimple element $s$ whenever the rank of $S$ is at most $r_0$ and $q \geq Q$. 

Next, view $S$ as $G/\bfZ(G)$ for $G := \cG^F$, and consider an inverse image $g \in G$ of $x$ in 
$G$ that is regular semisimple. Note that $|\bfC_G(g)| \leq (q+1)^r$ and so 
$|\bfC_G(x\bfZ(G))| \leq (q+1)^r|\bfZ(G)|$. Also, $|G| > (q-1)^{3r}$ and $|\bfZ(G)| \leq r_0+1$. Therefore,
$$|s^S| = \frac{|S|}{|\bfC_S(x)|} = \frac{|G|}{|\bfC_G(x\bfZ(G))|} \geq \frac{|G|}{(q+1)^r(r_0+1)} 
   > |S|^{3/5} >  4e|S|^{1/2} \log^{1/2}|S|$$
when $q \geq Q$ and we choose $Q$ large enough.   
\end{proof}

Next we recall the following fact:

\begin{lem}\label{regular-bound}
For any $r_0$, there is a constant $C = C(r_0)$ such that
$$|\chi(s)| \leq C$$ 
for all finite simple group $S$ of Lie type of rank $\leq r_0$,
for all regular semisimple elements $s \in S$, and for all $\chi \in \Irr(S)$.
\end{lem}

\begin{proof}
Note that if $S$ is not a Suzuki-Ree group, then the statement is a direct consequence of \cite[Proposition 5]{GLL}. But in fact the same proof goes through in the case $S$ is a Suzuki-Ree
group.
\end{proof}

\begin{prop}\label{suzuki-ree}
Theorem \ref{main1} holds for Suzuki and Ree groups, with $S_1 = \{1\}$.
\end{prop}

\begin{proof}
Let $S = \tw2 B_2(q^2)$, $\tw2 G_2(q^2)$, or $\tw2 F_4(q^2)$.
By \cite[Proposition 6.4.1]{LST} and Corollary \ref{regular}, there exists $Q_1 = Q(w_1,w_2)$ such that 
$w_1(S)w_2(S) = S$, and $w_i(S)$ contains a regular semisimple
element $s_i$ satisfying the condition \ref{regular}(ii) for $i = 1,2$, whenever $q \geq Q_1$. By 
Lemma \ref{regular-bound}, there is some $C > 0$, independent from $q$, 
such that $|\chi(s_i)| \leq C$ for all
$\chi \in \Irr(S)$ and $i = 1,2$. We will now prove that there is some $B > 0$, independent from $q$,
such that 
\begin{equation}\label{s-r}
  \sum_{1_S \neq \chi \in \Irr(S)}\frac{|\chi(g)|}{\chi(1)} \leq \frac{B}{q}
\end{equation}  
for all $1 \neq g \in S$. Taking $q \geq \max(Q_1,2BC^2)$, we will achieve (\ref{bound1}).
  
First let $S = \tw2 B_2(q^2)$ with $q \geq \sqrt{8}$. 
The character table of $S$ is known; see, e.g., \cite{Bu}. In particular,
$\Irr(S)$ consists of $q^2+3$ characters: $1_S$, two characters of degree 
$q(q^2-1)/\sqrt{2}$, and all the other ones have degree $\geq (q^2-1)(q^2-q\sqrt{2}+1)$. Furthermore, 
$$|\chi(g)| \leq q\sqrt{2}+1$$
for all $1_S \neq \chi \in \Irr(S)$ and $1 \neq g \in S$. It follows that  
$$\sum_{1_S \neq \chi \in \Irr(S)}\frac{|\chi(g)|}{\chi(1)} \leq (q\sqrt{2}+1)
     \left(\frac{2\sqrt{2}}{q(q^2-1)} + \frac{q^2}{(q^2-1)(q^2-q\sqrt{2}+1)}\right) < \frac{5}{q}$$ 
as stated.    

Next suppose that $S = \tw2 G_2(q^2)$ with $q \geq \sqrt{27}$. 
The character table of $S$ is known, see e.g. \cite{Wa}. In particular,
$\Irr(S)$ consists of $q^2+8$ characters: $1_S$, one character of degree 
$q^4-q^2+1$, six characters of degree $\geq q(q^2-1)(q^2-q\sqrt{3}+1)/\sqrt{12}$,
and the remaining characters of degree $\geq q^6/2$. Furthermore, 
$|\chi(g)| \leq \sqrt{|\bfC_S(g)|} \leq q^3$ for all $1 \neq g \in S$. It follows that  
$$\sum_{1_S \neq \chi \in \Irr(S)}\frac{|\chi(g)|}{\chi(1)} \leq q^3
     \left(\frac{1}{q^4-q^2+1} + \frac{6\sqrt{12}}{q(q^2-1)(q^2-q\sqrt{3}+1)} + 
     \frac{q^2}{q^6/2}\right) < \frac{5}{q}$$ 
as stated.    

Suppose now that $S = \tw2 F_4(q^2)$ with $q \geq \sqrt{8}$.
The (generic) character table of $S$ is known in principle,
but not all character values are given explicitly in \cite{Chevie} (in particular, $10$ families of characters
are not listed therein).  On the other hand, according to \cite{FG} and \cite{Lu2}, $\Irr(S)$ consists of 
$q^4+4q^2+17$ characters: $\chi_0 := 1_S$, four characters $\chi_{1,2,3,4}$ of degree 
$$\chi_{1,2}(1) = q(q^4-1)(q^6+1)/\sqrt{2},$$ 
$$\chi_3(1) = q^2(q^4-q^2+1)(q^8-q^4+1),~~~\chi_4(1) = (q^2-1)(q^4+1)(q^{12}+1);$$
all the other ones have degree $> q^{20}/48$ (when $q \geq \sqrt{8}$). 
The orders $|\bfC_S(g)|$ are listed in \cite{Chevie}, in particular, $|\bfC_S(g)| < 2q^{30}$ when
$1 \neq g \in S$. It follows that $|\chi(g)| < \sqrt{|\bfC_S(g)|} < \sqrt{2}q^{15}$ and so 
\begin{equation}\label{ree1}
  \sum_{\chi_{0,1,2,3} \neq \chi \in \Irr(S)}\frac{|\chi(g)|}{\chi(1)} < 
    \frac{\sqrt{2}q^{15}(q^4+4q^2+12)}{q^{20}/48} + \frac{\sqrt{2}q^{15}}{(q^2-1)(q^4+1)(q^{12}+1)} 
    < \frac{144}{q}.
\end{equation}    
Among all nontrivial conjugacy classes of $S$, there are two classes $g_{1,2}^S$ with 
$$|\bfC_S(g_1)| =  q^{24}(q^2-1)(q^4+1),~~|\bfC_S(g_2)| = q^{20}(q^4-1),$$
and all the other ones have centralizers of order $<4q^{20}$, cf. \cite{Chevie}. Hence if 
$g \notin \{1\} \cup g_1^S \cup g_2^S$ then $|\chi_i(g)| < 2q^{10}$ and so
\begin{equation}\label{ree2}
  \sum_{\chi = \chi_{1,2,3}}\frac{|\chi(g)|}{\chi(1)} \leq \frac{3 \cdot 2q^{10}}{q(q^4-1)(q^6+1)/\sqrt{2}}
  < \frac{10}{q}.
\end{equation}
Finally, for $g = g_{1,2}$, using \cite{Chevie} one can check that 
$$|\chi_{1,2}(g)| \leq q(q^6-q^4+1)/\sqrt{2},~~~|\chi_3(g)| \leq q^8-q^4+q^2,$$
whence
\begin{equation}\label{ree3}
  \sum_{\chi = \chi_{1,2,3}}\frac{|\chi(g)|}{\chi(1)} \leq \frac{\sqrt{2}q(q^6-q^4+1)}{q(q^4-1)(q^6+1)/\sqrt{2}}
  + \frac{q^8-q^4+q^2}{(q^2-1)(q^4+1)(q^{12}+1)}
  < \frac{1}{q}.
\end{equation} 
All together, (\ref{ree1})--(\ref{ree3}) imply (\ref{s-r}) for $S = \tw2 F_4(q^2)$.
\end{proof}

\begin{prop}\label{bounded}
Theorem \ref{main1} holds for all (sufficiently large) finite non-abelian simple groups $S$ of Lie type of bounded rank, with $S_1 = \{1\}$.
\end{prop}

\begin{proof}
By Proposition \ref{suzuki-ree}, we may assume that $S$ is not a Suzuki or Ree group. 
Assume that $S$ is defined over $\FF_q$ and of rank $\leq r_0$. Then we view 
$S$ as $\cG^F/\bfZ(\cG^F)$ for some simple, simply connected algebraic group $\cG$, of rank
$r \leq r_0$, and some Steinberg endomorphism $F:\cG \to \cG$. 
According to \cite[Theorem 1.7]{LS2}, $w_1(S)w_2(S) = S$ when $q$ is large enough.  
By \cite[Corollary 5.3.3]{LST}, 
there exists a positive constant $\delta = \delta(w_1,w_2,r_0)$ 
such that for any $F$-stable maximal torus $\cT$ of $\cG$ and for $i = 1,2$,
$$|\cT^F \cap w_i(\cG^F)| \geq \delta|\cT^F| \geq \delta(q-1)^r.$$
On the other hand, part (3) of the proof of \cite[Theorem 2.1]{Lu1} shows that  
$\cT^F$ contains at most $2^r r^2 (q+1)^{r-1}$ non-regular elements. Hence, if we choose
$$q > \max(5,1 + 3^{r_0}r_0^2/\delta),$$
then $\cT^F \cap w_i(\cG^F)$ contains a regular semisimple element. Now we apply this observation
to a pair of $F$-stable maximal tori $\cT_1$, $\cT_2$ of $\cG$ that is {\it weakly orthogonal} in 
the sense of \cite[Definition 2.2.1]{LST} and get regular semisimple elements 
$s_i \in \cT^F \cap w_i(\cG^F)$ for $i = 1,2$. By \cite[Proposition 2.2.2]{LST}, if 
$\chi \in \Irr(\cG^F)$ is nonzero at both $s_1$ and $s_2$, then $\chi$ is unipotent (and so
trivial at $\bfZ(\cG^F)$). In this case, 
the results of \cite{DL} imply that $\chi(s_1)$ does not depend on
the particular choice of the element $s_1$ of given type, and similarly for $\chi(s_2)$. Also, 
$|s_i^S| \geq 4e|S|^{1/2}\log^{1/2}|S|$ if $q > \max(Q(w_1,r_0),Q(w_2,r_0))$, cf. 
Corollary \ref{regular}.

We claim that we can find such a pair $\cT_1,\cT_2$ such that there are $\kappa \leq 4$ characters 
$\chi \in \Irr(\cG^F)$ with $\chi(s_1)\chi(s_2) \neq 0$, and moreover
$|\chi(s_1)\chi(s_2)| = 1$ for all such $\chi$. Indeed, this can be done with 
$\kappa = 2$ for $\cG^F$ of type $A_r$ by \cite[Theorem 2.1]{MSW}, of type 
$\tw2 A_r$ by \cite[Theorem 2.2]{MSW}, of type $C_r$ by \cite[Theorem 2.3]{MSW}, 
of type $B_r$ by \cite[Theorem 2.4]{MSW}, of type $\tw2 D_r$ by \cite[Theorem 2.5]{MSW},
and of type $D_{2l+1}$ by \cite[Theorem 2.6]{MSW}. For type $D_{2l}$ we can get 
$\kappa = 4$ by using \cite[Proposition 2.3]{GT}. For the exceptional groups of Lie type,
we can get $\kappa = 2$ by using \cite[Theorem 10.1]{LM}. 

Now consider any nontrivial element $g \in S$.  
Certainly, if $\kappa = 2$, then these
characters are the trivial character and the Steinberg character $\St$ of $\cG^F$. Since 
$S$ is simple, $\St$ is faithful and so $|\St(g)| < \St(1)$. But $\St(g) \in \bbZ$ divides $\St(1)$, 
so we get $|\St(g)/\St(1)| \leq 1/2$ and 
$$\sum_{1_S \neq \chi \in \Irr(S)}\left|\frac{\chi(s_1)\chi(s_2)\bar\chi(g)}{\chi(1)}\right| 
    = \frac{|\St(g)|}{\St(1)} \leq 1/2,$$
as desired. Finally, assume $\kappa = 4$ (so $\cG^F$ is of type $D_{2l}$). By 
\cite[Theorem 1.2.1]{LST} we have 
$$\sum_{1_S \neq \chi \in \Irr(S)}\left|\frac{\chi(s_1)\chi(s_2)\bar\chi(g)}{\chi(1)}\right| 
    \leq 3q^{-1/481} < 1/2$$
if $q > 6^{481}$.          
\end{proof}

To deal with (classical) groups of unbounded rank, we recall
the notion of the {\emph{support} of an element of a classical group \cite[Definition 4.1.1]{LST}.
For $g\in GL_n(\FF)\subset GL_n(\bar\FF)$, the support is the codimension of the largest eigenspace of $g$ acting on $\FF^n$.
The support of any element in a classical group $G(\FF)$ is the support of its image under the natural representation $\rho\colon G(\bar \FF)\to GL_n(\bar\FF)$.
Most elements have large support; we have the following quantitative estimate:

\begin{lem}\label{support}
Let $S$ be a finite simple classical group of rank $r \geq 8$ and $B \geq 1$ any constant.
If $r \geq 8B+3$, then the set $S_1$ of elements
of support $< B$ can contain at most $|S|^{1/2}$ elements of $S$.
\end{lem}

\begin{proof}
We will bound the total number $N$ of elements $g$ of support $\leq B$ in 
$L = SL_n(q)$, $SU_n(q)$, $Sp_n(q)$, or $SO^\pm_n(q)$ 
(note that $S \hookrightarrow L/\bfZ(L)$). Let $V = \FF_q^n$, respectively 
$\FF_{q^2}^n$, $\FF_q^n$, $\FF_q^n$, denote the natural $L$-module. By the results in 
\cite[\S3]{FG},  the number of conjugacy classes in $L$ is less than $16q^r \leq q^{r+4}$. Since 
$B < n/2$, $g$ has a primary eigenvalue $\lambda \in \FF_q^\times$, respectively
$\lambda^{q+1} = 1$, or $\lambda = \pm 1$, cf. \cite[Proposition 4.1.2]{LST}. Moreover, 
one can show that 
$V$ admits a $g$-invariant decomposition $V = U \oplus W$ into a direct (orthogonal if 
$L \neq SL_n(q)$) sum of (non-degenerate if $L \neq SL_n(q)$) subspaces, with
$U \leq \Ker(g-\lambda \cdot 1_V)$ and $m := \dim(U) \geq n-2B$ (see \cite[Lemma 6.3.4]{LST} 
for the orthogonal case).

Consider the case $L = SL^\epsilon_n(q)$, with $\epsilon = +$ for 
$SL$ and $\epsilon = -$  for $SU_n(q)$. Then $\bfC_L(g)$ contains 
$SL^\epsilon_m(q)$. It follows that 
$$|g^L| \leq \frac{|SL^\epsilon_n(q)|}{|SL^\epsilon_m(q)|} < \frac{2q^{n^2-1}}{q^{m^2-1}/2}
    = 4q^{n^2-m^2} \leq q^{4nB+2},$$
as $n \geq m \geq n-2B$. Hence,
$$N \leq q^{n(4B+1)+3} \leq q^{(n^2-3)/2} \leq |S|^{1/2}.$$

Suppose now that $L = SO^\pm_n(q)$.  Then $\bfC_L(g)$ contains 
$SO^\pm_m(q)$. It follows that 
$$|g^L| \leq \frac{|SO^\pm_n(q)|}{|SO^\pm_m(q)|} < \frac{q^{n(n-1)/2}}{q^{m(m-1)/2}/2}
    = 2q^{(n-m)(n+m-1)/2+1} \leq q^{(2n-1)B+2},$$
and so
$$N \leq q^{B(2n-1)+r+6} \leq q^{(n(n-1)/2-1)/2} \leq |S|^{1/2}.$$

Consider the case $L = Sp_n(q)$, so $n=2r$ and $m$ are even.  Then $\bfC_L(g)$ contains 
$Sp_m(q)$. It follows that 
$$|g^L| \leq \frac{|Sp_n(q)|}{|Sp_m(q)|} < \frac{q^{n(n+1)/2}}{q^{m(m+1)/2}/2}
    = 2q^{(n-m)(n+m+1)/2+1} \leq q^{(2n+1)B+2},$$
and so
$$N \leq q^{B(2n+1)+r+6} \leq q^{(n(n+1)/2-1)/2} \leq |S|^{1/2}.$$
\end{proof}

\begin{thm}\label{clas}
Theorem \ref{main1} holds for all simple classical groups of sufficiently large rank. 
\end{thm}

\begin{proof}
(a) View $S = G/\bfZ(G)$ with $G = \cG^F$ as above and let $r := \rank(\cG)$. We will show that 
there are some $r_0 = r_0(w_1,w_2) > 8$ and $B = B(w_1,w_2)$ such that Theorem \ref{main1} holds
when $r \geq r_0$, for suitable regular semisimple elements $s_1,s_2 \in S$ and with 
$S_1$ being the set of elements in $S$ of support $< B$. By Lemma \ref{support},
$|S_1| \leq |S|^{1/2}$ if $r_0 \geq 8B+3$. 

Again, note that for any regular semisimple element $h \in G$, $\bfC_\cG(h)$ is a maximal torus
(as $\cG$ is simply connected) and so $|\bfC_G(h)| \leq (q+1)^r$. It follows that 
$|\bfC_G(h\bfZ(G))| \leq (q+1)^r|\bfZ(G)|$ and so $|\bfC_S(h\bfZ(G))| \leq (q+1)^r$. Also,
$|G| > q^{r(r+1)}$ and $|\bfZ(G)| \leq r+1$. So when $r \geq r_0 > 8$ we have 
$$|\bfC_S(h\bfZ(G))| \leq (q+1)^r < \left(\frac{q^{r(r+1)}}{r+1}\right)^{1/3} < |S|^{1/3}.$$  
In particular, $s_1$ and $s_2$ satisfy condition (iii) of Theorem \ref{main1} when $r_0 \geq 9$. 
As mentioned above, condition (i) of Theorem \ref{main1}  follows from 
\cite[Theorem 1.1.1]{LST}. So it suffices to establish (\ref{bound1}) for all $g \in S \setminus S_1$.

\smallskip
(b) Suppose first that $\cG^F$ is a special linear, special unitary, or symplectic group. By 
Propositions 6.2.4 and 6.1.1 of \cite{LST}, there is some $r_1 = r_1(w_1,w_2)$ with the following 
property. When $r \geq r_1$, there are regular semisimple elements $s_i \in w_i(S)$ for $i = 1,2$ such 
that there are at most $\kappa \leq 4$ 
irreducible characters $\chi_i \in \Irr(S)$ with $\chi_i(s_1)\chi_i(s_2) \neq 0$, 
$1 \leq i \leq \kappa$, and $\chi_1 = 1_S$. Moreover, 
$|\chi_i(s_1)\chi_i(s_2)| = 1$ for $1 \leq i \leq \kappa$. Now we choose $B \geq 1443^2$ and 
consider any $g \in S \setminus S_1$. By \cite[Theorem 1.2.1]{LST}, 
$$\frac{|\chi(g)|}{\chi(1)} < q^{-\sqrt{B}/481} < q^{-3} \leq 1/8,$$
whence
$$\left|\sum_{1_S \neq \chi \in \Irr(S)}\frac{\chi(s_1)\chi(s_2)\bar\chi(g)}{\chi(1)}\right|
    \leq \sum^\kappa_{i=2}\frac{|\chi_i(g)|}{\chi_i(1)} < 3/8,$$
as required. In fact, if $\cG^F$ is a symplectic group, then $\kappa = 2$, 
$\chi_2 = \St$, $|\chi_2(g)/\chi(1)| \leq 1/q \leq 1/2$ for all $1 \neq g \in S$ and so
we can take $S_1 = \{1\}$. 

\smallskip
(c) Suppose now that $\cG^F$ is a simple orthogonal group. By 
Propositions 6.3.5 and 6.3.7 of \cite{LST}, there are some $r_2 = r_2(w_1,w_2)$,
$\kappa = \kappa(w_1,w_2)$, and $C = C(w_1,w_2)$ with the following property. When $r \geq r_2$, there are regular semisimple elements $s_i \in w_i(S)$ for $i = 1,2$ such that there are at most 
$\kappa$ irreducible characters $\chi_i \in \Irr(S)$ with $\chi_i(s_1)\chi_i(s_2) \neq 0$, 
$1 \leq i \leq \kappa$, and $\chi_1 = 1_S$. Moreover, 
$|\chi_i(s_1)\chi_i(s_2)| \leq C$ for $1 \leq i \leq \kappa$. Now we choose $B \geq 1443^2$ such that
$$(\kappa-1)C^2 2^{-\sqrt{B}/481} < 1/2.$$
Then  for any $g \in S \setminus S_1$, by \cite[Theorem 1.2.1]{LST} we have  
$$\left|\sum_{1_S \neq \chi \in \Irr(S)}\frac{\chi(s_1)\chi(s_2)\bar\chi(g)}{\chi(1)}\right|
    \leq \sum^\kappa_{i=2}\frac{C^2|\chi_i(g)|}{\chi_i(1)} < (\kappa-1)C^2  2^{-\sqrt{B}/481} < 1/2.$$
Hence we are done by choosing $r_0 := \max(r_1,r_2,9,8B+3)$.
\end{proof}

\section{Alternating Groups}

\newcommand\N{\mathbb{N}}
\newcommand\Fix{\mathrm{Fix}}

Suppose that $G$ is a group and $X$ and $Y$ are subsets.  If we have subsets $X_1,\ldots,X_k\subseteq X$, $Y_i,\ldots,Y_k\subseteq Y$, and
$Z_1,\ldots,Z_k\subseteq Z$ such that $Z_i\subseteq X_i Y_i$ and $\bigcup Z_i = G$, then setting $X_0 = X_1\cup\cdots \cup X_k$ and $Y_0 = Y_1\cup\cdots\cup Y_k$,
we have $X_0 Y_0 = G$.  We use this construction to find $X_0\subseteq w_1(\AN)$ and $Y_0\subseteq w_2(\AN)$ such that $X_0 Y_0 = \AN$ and
$|X_0|, |Y_0|$ are of order $n!^{1/2} \sqrt{\log n!}$.

We begin by noting that for any word $w$ and any group $G$, $w(G)$ is a characteristic set, i.e. invariant under every automorphism of $G$.
In particular, $w(\AN)$ is a union of $\SN$-conjugacy classes.  If $g_1,g_2\in \AN$ and $C_1$ and $C_2$ denote their $\SN$-conjugacy classes, then
\begin{equation}
\label{Schur}
|\{(c_1,c_2)\in C_1\times C_2\mid c_1c_2 = g\}| = \frac{|C_1|\,|C_2|}{n!} \sum_\chi \frac{\chi(g_1)\chi(g_2)
\bar\chi(g)}{\chi(1)}.
\end{equation}

We recall a basic upper bound estimate \cite[Theorem 1.1]{LS1} for $|\chi(g)|$.  For $g\in \SN$ and $i\in \N$, let $\Sigma_i(g)$ denote the union of all $g$-cycles of length $\le i$
in $\{1,\ldots,n\}$.  Define $e_1(g),e_2(g),\ldots$ so that
$$n^{e_1(g)+\cdots+e_i(g)} = \max(1,|\Sigma_i(g)|)$$
for all $i\in\N$.  Define
$$E(g) = \sum_{i=1}^\infty \frac{e_i(g)}i.$$
Then for all $\epsilon > 0$ there exists $N$ such that for all $n>N$, all $g\in \SN$, and all irreducible characters $\chi$ of $\SN$,
$$|\chi(g)| \le |\chi(1)|^{E(g)+\epsilon}.$$
For example, if $g$ has a bounded number of cycles, and $n$ is sufficiently large in terms of $\epsilon$,
$$|\chi(g)| \le |\chi(1)|^\epsilon.$$
If $g$ has no more than $n^{2/3}$ fixed points and $n$ is sufficiently large in terms of $\epsilon$,
then
$$|\chi(g)| \le |\chi(1)|^{5/6+\epsilon}.$$

By a result of Liebeck and Shalev \cite[Theorem 1.1]{LiS}, for all $s>0$,
$$\lim_{n\to \infty} \sum_{\chi \in \Irr(\SSS_n)} \chi(1)^{-s} = 2.$$
Note that the trivial character and the sign character each contribute $1$ to the above sum; excluding them from the sum, the limit would be zero.
Of course, thus if $g_1$, $g_2$, and $g$ are all even permutations, then the trivial character and the sign character each contribute
$\frac{|C_1| |C_2|}{n!}$ to the expression (\ref{Schur}).  From this, we conclude:
\begin{prop}
\label{summary}
For all $\epsilon >0$ and integers $k_1$ and $k_2$, there exists an integer $N = N(\epsilon,k_1,k_2)$ such that if $n>N$ and $C_1$ and $C_2$ are even conjugacy
classes in $\SN$ consisting of $k_1$ and $k_2$ cycles respectively, then every $g\in \AN$ with 
no more than $n^{2/3}$ fixed points is represented in at least
$$(1-\epsilon)\frac{|C_1|\,|C_2|}{|\AN|}$$
different ways as $x_1 x_2$, $x_1\in C_1$, $x_2\in C_2$.
\hfill $\Box$
\end{prop}

Now, by \cite[Theorem 1.3]{LS2}, if $n$ is sufficiently large, $w_1(\AN)$ and $w_2(\AN)$ each contain elements $g_1$ and $g_2$ respectively, with at most $6$ cycles of length $>1$ and 
$\leq 17$ cycles in total. So there is some constant $A$ such that $|\bfC_{\SSS_n}(g_i)| < An^6$ 
for $i = 1,2$, whence 
$$|w_i(\AN)| \geq |(g_i)^{\SN}| > 2e(n!)^{1/2}\log^{1/2}n!.$$
Defining $Z_1$ as the set of elements of $\AN$ with no more than $n^{2/3}$ fixed points,
it follows from Proposition~\ref{main-prop}, that there exist $X_1$ and $Y_1$  
contained in $w_1(\AN)$ and $w_2(\AN)$ respectively, such that
$Z_1\subseteq X_1 Y_1$.  

\medskip
What remains is to define $X_i,Y_i,Z_i$ for $i\ge 2$ to cover the elements of $\AN$ with more than $n^{2/3}$ fixed points.

The number of elements of $\AN$ with at least $m := \lceil2n/3\rceil$ fixed points is less than
$$\sum_{i=m}^n \binom ni (n-i)! = \sum_{i=m}^n \frac{n!}{i!} < 2\frac{n!}{m!} \le n!^{1/3+o(1)}.$$
Therefore, we can represent each element $g$ with at least $m$ fixed points as $x_g y_g$, $x_g\in w_1(\AN)$, $y_g \in w_2(\AN)$, and we can define $X_2$ to be the union of all such $x_g$
and $Y_2$ the union of all such $y_g$.  Note that 
$$|X_2|, |Y_2| < (n!)^{1/3+o(1)}.$$
This reduces the problem to elements $g$ with 
$$n^{2/3}\le |\Fix(g)|\le 2n/3.$$

For each $T\subseteq \{1,2,\ldots,n\}$ with $m := |T|\in [n^{2/3},2n/3]$, we define 
$\SSS_T\subseteq \SN$ to be the pointwise stabilizer of $T$ in $\SN$ 
and $\AAA_T$ to be the pointwise stabilizer of $T$ in $\AN$.
Thus $\SSS_T$ is isomorphic to $\SSS_{n-m}$ and $\AAA_T$ is isomorphic to 
$\AAA_{n-m}$,
where $n-m\in [n/3,n-n^{2/3}]$.  For each $T$, we choose an $\SSS_T$-conjugacy class $C_{1,T}$ in $w_1(\AAA_T)$ and an $\SSS_T$-conjugacy class $C_{2,T}$ in $w_2(\AAA_T)$, each consisting of at most $17$ cycles when regarded as elements of $\SSS_{n-m}$.  (Of course there are $|T|$ additional $1$-cycles when we regard them as elements of $\SN$.)
If $n$ is sufficiently large, $n-m$ is larger than the constant $N$ of Proposition~\ref{summary}, and we conclude that every fixed point free element of $\AAA_{n-m}$ can be written in 
at least
$$(1-\epsilon)\frac{|C_{1,T}|\,|C_{2,T}|}{|\AAA_{n-m}|}$$
ways.  Applying Proposition~\ref{main-prop} and arguing as above, we conclude that there exist subsets $X_T$ and $Y_T$ of $C_{1,T}$ and $C_{2,T}$ respectively such that 
$X_T Y_T$ contains all elements of $\SN$ with fixed point set exactly $T$, and $|X_T|$ and $|Y_T|$ are bounded above by 
$$c (n-m)!^{1/2}\log^{1/2}(n-m)!,$$
where
$c$ is independent of $n$ or $m$.
An upper bound for the cardinality of $\bigcup_T X_T$  is
$$cn\log n\sum_{n^{2/3}\le m\le 2n/3} \binom nm (n-m)!^{1/2}\le cn^3\max\Bigl\{\binom nm (n-m)!^{1/2}\Bigm| n^{2/3}\le m\le 2n/3\Bigr\},$$
and likewise for $\bigcup_T Y_T$.

For $m\ge n^{2/3}$, we have by Stirling's approximation
$$m! > (m/e)^m.$$
So when $n>(2e^2)^3$ is large enough, we have that
$$\frac{\binom{n}{m} \cdot (n-m)!^{1/2}}{(n!)^{1/2}} = 
    \frac{(\prod^n_{j=n-m+1}j)^{1/2}}{m!} < \frac{n^{m/2}}{e^{-m}m^m}$$
$$ = 
    \left(\frac{e^2n}{m^2}\right)^{\frac{m}{2}} < \left(\frac{e^2}{n^{1/3}}\right)^{\frac{n^{2/3}}{2}} < 
    \left(\frac{1}{2}\right)^{\frac{n^{2/3}}{2}} < 
    \frac{1}{cn^3}.$$
In this case, the cardinalities of $\bigcup_T X_T$ and $\bigcup_T Y_T$ are less than 
$n!^{1/2}$.
It follows that $X_1$, $X_2$, and all the $X_T$ together have cardinality $O((n!)^{1/2}\log^{1/2}n)$,
and likewise for $Y$.  That concludes the proof of Theorem~\ref{main} in the alternating case.

\section{Groups as Products of Two Subsets}

\begin{lem}\label{cyclic}
Let $G$ be a cyclic group of prime order $p$ and $x$ any real number with 
$2 \le x\le p$.  Then there exist subsets $X$ and $Y$ of $G$ with 
$|X| \le x$ and $|Y| \le 2p/x$ such that $XY = G$.
\end{lem}

\begin{proof}
Identify $G$ with the additive group $\bbZ/p\bbZ$ and its elements with $0, 1, \ldots, p-1$.
The cases $2 \leq p \leq 7$ are obvious, so we will assume $p \geq 11$. Since the roles of 
$x$ and $2p/x$ are symmetric, we may assume that $x \geq \sqrt{2p} > 4$. Now if 
$x \geq p-2$ then $G = X+Y$ with $X := \{2j  \mid 0 \leq j \leq (p-1)/2\}$ and $Y = \{0,1\}$. 
Suppose $p-2 > x \geq \sqrt{2p}$. Setting $a := \lfloor x \rfloor \leq x$ and 
$b := \lceil p/a \rceil \geq p/a$, we see that $b < \max(p/a+1,2p/x)$ and $G = X+Y$ for 
$$X := \{0,1, \ldots ,a-1\},~~Y = \{ja \mid 0 \leq j \leq b-1\}.$$  
\end{proof}

\begin{lem}\label{max}
Let $G$ be a finite non-abelian simple group of order $n$. Then $G$ possesses
a maximal subgroup $M$, with $|M| \geq \sqrt{n}$ if $G = J_3$ and 
$|M| \geq \sqrt{2n}$ otherwise.
\end{lem}

\begin{proof}
The case of $26$ sporadic simple groups can be checked using \cite{Atlas}.
If $G = \AAA_n$ with $n \geq 5$, take $M := \AAA_{n-1}$. So we may assume
that $G$ is a finite simple group of Lie type. If $G$ is a classical group,
then the smallest index of proper subgroups of $G$ is listed in
\cite[Table 5.2.A]{KL}, whence the statement follows. If $G$ is an 
exceptional group, then Table 3.5 of \cite{MMT} lists a subgroup 
$N$ of $G$, and one can check that $|N| \geq \sqrt{2n}$. 
\end{proof}

\begin{thm}\label{product}
Let $G$ be any finite group of order $n$ and $x$ any real number with $2 \le x\le n$.  Then there exist subsets $X$ and $Y$ of $G$ with $|X| \le x$ and $|Y| \le 2n/x$ such that $XY = G$.
\end{thm}

\begin{proof}
We proceed by induction on $|G|$. Note that the roles of $x$ and $y := 2n/x$ 
in the statement are symmetric, and so without loss we may assume 
$x \leq y$, i.e. $x \leq \sqrt{n/2}$. 

\smallskip
(a) Suppose that there is a subgroup $H < G$ with $|H| > x$. By the induction 
hypothesis, there exist subsets $X', Y' \subseteq H$ with $X'Y'= H$, 
$|X'| \leq x$, and $|Y'| \leq 2 |H|/x$. Decompose $G = \bigcup^m_{i=1}Hy_i$ 
with $m = [G:H]$, and let $X := X'$ and $Y := \bigcup^m_{i=1} Y'y_i$. Then 
$XY = G$, $|X| \leq x$, and $|Y| \leq m|Y'| \leq 2|G|/x$.

Next, let us consider the possibility that $H < G$ is a subgroup with
$x/2 \leq |H| < x$.  Then setting $X := H$ and $Y$ a set of coset 
representatives of $H$ in $G$, we get $G = XY$, $|X| \leq x$, and
$|Y| = [G:H] \leq  2n/x$.

Thus we are done if $G$ possesses a proper subgroup of 
order $\geq x/2$. 

\smallskip
(b) Suppose now that $G$ admits a nontrivial {\it normal} subgroup $H$ 
with $|H| < x/2$. By the induction hypothesis applied to $G/H$ and 
$x' := x/|H|$, there exist subsets $X', Y' \subseteq G/H$ with 
$|X'| \leq x'$, $|Y'| \leq 2|G/H|/x' = 2n/x$, and $X' Y' = G/H$. Now let
$X$ denote the full inverse image of $X'$ in $G$, and let $Y$ denote a
set of coset representatives in $G$ for $Y'$. Then $G = XY$,
$|X| = |X'|\cdot|H| \leq x$, and $|Y| = |Y'| \leq 2n/x$. 

\smallskip
(c) Assume $G$ is not simple: $1 \neq N \lhd G$ for some $N < G$. If 
$|N| \geq x/2$ then we are done by (a). Otherwise we are done by (b).

It remains to consider the case $G$ is simple. If $G$ is abelian, then we
can apply Lemma \ref{cyclic}. Otherwise by Lemma \ref{max} there is 
a maximal subgroup $M < G$ of order $\geq \sqrt{n} > x/2$, and so we are again
done by (a).
\end{proof}

\begin{cor}\label{root}
Any finite group $G$ admits a {\rm square root} $R$, i.e. a subset $R \subseteq G$
such that $R^2 = G$, with $|R| \leq \sqrt{8|G|}$.
\end{cor}

\begin{proof}
Taking $x = \sqrt{2|G|}$ in Theorem \ref{product}, we see that $G = XY$ with 
$|X|, |Y| \leq x$. Now set $R := X \cup Y$.
\end{proof}

\section{Square Roots of a Lie Group}

\newcommand\Mdim{\overline{\dim}\,}
In this section we show that the results of \S5 extend in a suitable sense to compact Lie groups. 
We would like to say that the minimum dimension of a square root of $G$ is half the dimension of $G$, but
we need a suitable definition of dimension.
Hausdorff dimension does not do the job;
indeed, it is not difficult to see that $S^1$ can be written as $XY$, where $X$ and $Y$ are both of Hausdorff dimension $0$.
It turns out that upper Minkowski dimension is the better notion for our purposes.

We begin by recalling some basic definitions.  A good  reference is \cite{Tao}.
For $\delta > 0$, we define the \emph{$\delta$-packing number} of a bounded 
metric space $X$, $N_\delta(X)$, to be the maximum number of 
disjoint open balls of radius $\delta$ in $X$.  We recall that the \emph{upper Minkowski dimension}, $\Mdim X$, of a bounded metric space $X$, is given by the formula
$$\Mdim X = \limsup_{\delta > 0} \frac{-\log N_\delta(X)}{\log \delta}.$$
If $\phi\colon X\to Y$ is a surjective Lipschitz map with constant $L$, then $N_{L\delta}(Y) \le N_\delta(X)$, so $\Mdim \phi(X)\le \Mdim X$.

If $[-1,1]$ is endowed with the usual metric $d(x,y) = |x-y|$, then
$$N_\delta([-1,1]) = \lfloor 1/\delta\rfloor,$$
and it follows that $\Mdim [-1,1] = 1$.
If the ring $\ZZ_p$ of $p$-adic integers is endowed with the usual metric $d(x,y) = |x-y|_p$, it follows that
$$N_\delta(\ZZ_p) = p^{\max(0,1+\lfloor-\log_p\delta \rfloor)},$$
so $\Mdim \ZZ_p = 1$.

Upper Minkowski dimension is well suited to our purposes because of the following elementary proposition, which is well known 
for subsets of Euclidean spaces \cite[8.10--8.11]{Mattila}.

\begin{prop}
\label{Mdim-product}
Let $(X,d_X)$ and $(Y,d_Y)$ be bounded metric spaces, and let $d$ be a metric on $X\times Y$ such that
\begin{equation*}
\label{product-metric}
\max (d_X(x_1,x_2),d_Y(y_1,y_2)) \le d((x_1,y_1),(x_2,y_2)) \le d_X(x_1,x_2)+d_Y(y_1,y_2).
\end{equation*}
Then
\begin{equation}
\label{upper}
\Mdim X\times Y \le \Mdim X + \Mdim Y,
\end{equation}
with equality if $\log N_\delta(X)/\log \delta$ and $\log N_\delta(Y)/\log \delta$ both converge as $\delta\to 0$.
\end{prop}

\begin{proof}
If $x_1,\ldots,x_m$ are the centers of a maximal collection of disjoint open balls of radius $\delta$ in $X$,
then balls of radius $2\delta$ centered at $x_1,\ldots,x_m$ cover $X$, and likewise for $Y$.
The product of any ball of radius $2\delta$ in $X$ and any ball of radius $2\delta$ in $Y$ is contained in some ball of radius $4\delta$ in $X\times Y$, so $X\times Y$ can be covered 
by $N_\delta(X) N_\delta(Y)$ balls of radius $4\delta$.  
Given any disjoint collection of balls of radius $4\delta$ in $X\times Y$, 
no two centers can lie in the same ball of radius $4\delta$.  Thus, 
$$N_{4\delta}(X\times Y) \le N_\delta(X) N_\delta(Y),$$
which proves (\ref{upper}).
On the other hand, if $x_1,\ldots,x_m$ are centers of disjoint balls of radius $\delta$ in $X$
and $y_1,\ldots,y_n$ are centers of disjoint balls of radius $\delta$ in $Y$, then
$(x_i,y_j)$ are the centers of disjoint balls of radius $\delta$ in $X\times Y$, so
$$N_{\delta}(X\times Y) \ge N_\delta(X) N_\delta(Y).$$
It follows that
$$\lim_{\delta\to 0} \frac{-\log N_\delta(X\times Y)}{\log \delta} 
= \lim_{\delta\to 0} \frac{-\log N_\delta(X)}{\log \delta} + \lim_{\delta\to 0} \frac{-\log N_\delta(Y)}{\log \delta}$$
if both limits on the right hand side exist.
\end{proof}

Now let $G$ be a compact Lie group.  
We say that a metric $d$ on $G$ is \emph{compatible} if it is left- and right-invariant by $G$ 
and there exists a coordinate map from some open neighborhood of the identity $e$ of $G$ to some open set in $\RR^n$ which
is Lipschitz in some neighborhood of $e$.  If this is true for some coordinate map, it is true for all coordinate maps at $e$, since
smooth maps between open sets in $\RR^n$ are locally Lipschitz.  Likewise, a compatible metric on a 
compact $p$-adic Lie group is
a translation-invariant metric for which there exists a coordinate map from some open neighborhood of $e$ to some open set in $\QQ_p^n$, and the choice of coordinate map does not matter.  We recall \cite[III, \S4, no.\ 3]{Bourbaki} that every real (resp. $p$-adic) Lie group admits an \emph{exponential} map from a neighborhood of $0$ in $\RR^n$ (resp. $\QQ_p^n$) which is bijective and whose inverse is a coordinate map.

\begin{prop}
Let $G$ be a compact Lie group endowed with a compatible metric.  Then $\Mdim G$ coincides 
with the usual topological dimension of $G$.
\end{prop}

\begin{proof}
By Proposition~\ref{Mdim-product}, $\Mdim I^n = n$, where $I$ is any open interval in $\RR$,
and it follows that $\Mdim U = n$
for any bounded open set in $\RR^n$.  If $\phi\colon U\to G$ is a bi-Lipschitz coordinate map, then $U' := \phi(U)$ is an open subset of $G$ of dimension $n$.  Therefore, any translate of $U'$ in $G$ has dimension $n$, and likewise for any finite union of such translates.  By compactness, $G$ itself is such a union, so
$\Mdim G = \dim G$.
\end{proof}

There is also a $p$-adic version of the same proposition, whose proof is the same:

\begin{prop}
Let $G$ be a compact $p$-adic Lie group endowed with a compatible metric.  Then $\Mdim G$ coincides 
with the usual topological dimension of $G$.
\end{prop}

We can now prove our lower bound for square roots of a real or $p$-adic Lie group.

\begin{prop}\label{compact1}
If $X$ and $Y$ are subsets of a compact real or $p$-adic Lie group $G$ endowed with a compatible metric $d$ 
and $XY=G$, then $\Mdim X+\Mdim Y \ge \dim G$.
In particular, if $X$ is a square root of $G$, $\Mdim X\ge (\dim G)/2$.
\end{prop}

\begin{proof}
Defining the metric $e$ on $G\times G$ by
$$e((g_1,h_1),(g_2,h_2)) := d(g_1,g_2)+d(h_1,h_2),$$
we have
$$d(g_1 h_1,g_2 h_2) \le d(g_1 h_1, g_1 h_2) + d(g_1 h_2, g_2 h_2) = e((g_1,h_1),(g_2,h_2)).$$
Thus, the multiplication map $m\colon G\times G\to G$ is Lipschitz.  It follows that
$$\Mdim XY = \Mdim m(X\times Y) \le \Mdim X\times Y \le \Mdim X+\Mdim Y.$$
If $XY = G$, then 
$$\Mdim X+\Mdim Y \ge \Mdim G = \dim G.$$
\end{proof}

The more interesting direction is the converse:

\begin{thm}\label{compact2}
Let $G$ be a compact real or $p$-adic Lie group, endowed with a compatible metric.  Then
$G$ has a square root of dimension $(\dim G)/2$.
\end{thm}
 
\begin{proof}

Let $G$ be a real (resp. $p$-adic) Lie group,  $L$ the Lie algebra, and $\exp$ the exponential map from a neighborhood $U$ of $0$ in $L$ to a neighborhood $N$ of $e\in G$. Let $v\in L$ be a sufficiently small non-zero element, specifically, an element satisfying $[-1,1]v\subset U$
(resp. $\ZZ_p v\subset U$).  Then the function $e_v\colon [-1,1]\to G$ (resp. $e_v\colon \ZZ_p\to G$) defined by
$e_v(t) = \exp(tv)$ is Lipschitz.  Let $C_v$ denote the image of $e_v$.

Choose a basis $v_1,\ldots,v_n$ of sufficiently small vectors in $L$.
If $n = 2k$,  let $X_0 = C_{v_1}\cdots C_{v_k}$ and $Y = C_{v_{k+1}}\cdots C_{v_{2k}}$.  As $X_0$ and $Y$ are each images of sets of dimension $k$ under Lipschitz maps, 
$\Mdim X_0,\Mdim Y\le k = (\dim G)/2$.
On the other hand, $X_0 Y$ contains a neighborhood of $e$ in $G$, so letting $X$ denote a suitable finite union of left translates of $X$, we have $XY = G$ and $\Mdim X\le k$.  Thus $X\cup Y$ is a square root of $G$ of dimension $(\dim G)/2$.

If $n = 2k+1$, we  observe that there exist subsets $A$ and $B$ of $[-1,1]$ such that $\Mdim A = \Mdim B = 1/2$ and
$A+B = [-1,1]$.  We can take, for instance, the Cantor sets
$$A = -a_0 + \sum_{i=1}^\infty  a_i 4^{-i}\,,a_i\in \{0,1\};\ B = \sum_{i=1}^\infty b_i4^{-i},\,b_i\in \{0,2\}.$$
Likewise, there exist $A,B\subset \ZZ_p$ of dimension $1/2$ such that $A+B = \ZZ_p$, for instance,
$$A = \sum_{i=1}^\infty  a_i p^{2i},\,a_i\in \{0,1,\ldots,p-1\};\ B = \sum_{i=1}^\infty b_ip^{2i},\,b_i\in \{0,p,2p,\ldots,(p-1)p\}.$$
Now, setting 
$$X_0 = C_{v_1}\cdots C_{v_k} \exp(A v_{k+1}),\ Y =\exp(B v_{k+1}) C_{v_{k+2}}\cdots C_{v_{2k+1}},$$
we see that
$$X_0 Y = C_{v_1}\cdots C_{v_{2k+1}}$$
contains a neighborhood of $e$, while $\Mdim X_0, \Mdim Y\le k+1/2.$
The rest of the argument goes as before.
\end{proof}

\end{document}